\documentclass[12pt]{article}

\usepackage{amsmath}
\usepackage{amsthm}
\usepackage{amsfonts}
\usepackage{amssymb}
\usepackage{graphicx}
\usepackage[left=3cm,right=3cm,top=3cm,bottom=3.5cm]{geometry}
\date{February 29, 2016}
\newcommand{\Z}{\ensuremath{\mathbb{Z}}}
\newtheorem{theorem}{Theorem}[section]
  \newtheorem{lemma}[theorem]{Lemma}
  \newtheorem{corollary}[theorem]{Corollary}
  \newtheorem{conjecture}[theorem]{Conjecture}
  \newtheorem{definition}[theorem]{Definition}
  \newtheorem{proposition}[theorem]{Proposition}

\begin{document}

\title{Decomposing highly edge-connected graphs into homomorphic copies of a fixed tree}

\author{Martin Merker\footnote{Department of Applied Mathematics and Computer Science, Technical University of Denmark, DK-2800 Lyngby, Denmark. E-mail address: marmer@dtu.dk}}
\maketitle

\begin{abstract}
The Tree Decomposition Conjecture by Bar\'at and Thomassen states that for every tree $T$ there exists a natural number $k(T)$ such that the following holds: If $G$ is a $k(T)$-edge-connected simple graph with size divisible by the size of $T$, then $G$ can be edge-decomposed into subgraphs isomorphic to $T$. So far this conjecture has only been verified for paths, stars, and a family of bistars. We prove a weaker version of the Tree Decomposition Conjecture, where we require the subgraphs in the decomposition to be isomorphic to graphs that can be obtained from $T$ by vertex-identifications. We call such a subgraph a homomorphic copy of $T$. This implies the Tree Decomposition Conjecture under the additional constraint that the girth of $G$ is greater than the diameter of $T$. As an application, we verify the Tree Decomposition Conjecture for all trees of diameter at most 4.
\end{abstract}

\section{Introduction}
An $H$-decomposition of a graph $G$ is a partition of its edge-set into subgraphs isomorphic to $H$. In 2006, Bar\'at and Thomassen~\cite{DecompConjecture} made the following conjecture:

\begin{conjecture}\label{conj-treedecomp}
For every tree $T$ on $m$ edges, there exists a natural number $k(T)$ such that the following holds: \par 
If $G$ is a $k(T)$-edge-connected simple graph with size divisible by $m$, then $G$ has a $T$-decomposition.
\end{conjecture}

The conjecture trivially holds if $T$ is a single edge. It is easy to see that edge-connectivity 1 suffices for a decomposition into paths of length 2, see for example~\cite{s-partitions} or~\cite{4paths}. When the conjecture was made, these were the only two cases for which it was known to be true. Thomassen verified the conjecture for paths of length 3 in~\cite{3paths} and for paths whose length is a power of 2 in~\cite{4paths} and~\cite{2kpaths}. Botler, Mota, Oshiro, and Wakabayashi proved it for the path of length 5~\cite{5path}, and also extended the result to paths of any given length~\cite{allpathsBrazil}. Another proof of the conjecture for paths of any length was found by Bensmail, Harutyunyan, and Thomass\'e~\cite{allpaths}. 

The results on the weak $k$-flow conjecture in~\cite{weak3flow} imply that the conjecture holds for all stars. It was further verified for the bistar on 4 edges with degree sequence $(3,2,1,1,1)$ by Bar\'at and Gerbner in~\cite{barat}. More generally, Thomassen proved the conjecture for all bistars where the degrees of the two non-leaves differ by precisely 1 in~\cite{bistars}. 

The aim of this paper is to prove a weaker version of Conjecture~\ref{conj-treedecomp} with a less restrictive notion of $H$-decompositions. Let us say that $H$ is a \textit{homomorphic copy} of $G$ if it can be obtained from $G$ by identifying some of its vertices and keeping all edges. Equivalently, $H$ is a homomorphic copy of $G$ if there exists a homomorphism from $G$ to $H$ that is bijective on the edge sets. In particular, a homomorphic copy of $G$ always has the same size as $G$, and every graph isomorphic to $G$ is also a homomorphic copy of $G$. 

\begin{definition}
Let $H$ and $G$ be graphs. An \textit{$H^*$-decomposition} of $G$ is a partition of the edge-set of $G$ into homomorphic copies of $H$.
\end{definition}

Clearly every $H$-decomposition is also an $H^*$-decomposition. If $T$ is a tree, then a $T^*$-decomposition is also a $T$-decomposition provided that the graph we decompose has large girth:
If $T'$ is a homomorphic copy of a tree $T$, but not isomorphic to $T$, then there exist two vertices $u,v$ in $T$ that have the same image in $T'$. Since we require the homomorphism to preserve all edges, the path between $u$ and $v$ gets mapped to a closed walk in $T'$. In particular, there exists a cycle in $T'$ whose length is at most the distance between $u$ and $v$ in $T$. Thus, if the girth of $G$ is at least the diameter of $T$, then every $T^*$-decomposition of $G$ is also a $T$-decomposition.

The following weakening of Conjecture~\ref{conj-treedecomp} is the main result of this paper.

\begin{theorem}\label{thm-treedecomp}
For every tree $T$ on $m$ edges, there exists a natural number $k_{h}(T)$ such that the following holds: 

If $G$ is a $k_{h}(T)$-edge-connected graph with size divisible by $m$, then~$G$ has a $T^*$-decomposition.

In particular, if $d$ is the diameter of $T$, then every $k_h(T)$-edge-connected graph $G$ with girth greater than $d$ has a $T$-decomposition.
\end{theorem}

Notice that we also allow graphs with multiple edges in Theorem~\ref{thm-treedecomp}, while Conjecture~\ref{conj-treedecomp} only holds for simple graphs: The graph $nK_2$ consisting of two vertices joined by $n$ edges has edge-connectivity $n$, but the only tree $T$ for which it has a $T$-decomposition is $K_2$, a single edge. However, the graph $nK_2$ is itself a homomorphic copy of any tree on $n$ edges, and it has a $T^*$-decomposition for any tree $T$ whose size divides $n$.

The main ingredients of the proof of Theorem~\ref{thm-treedecomp} are the results on the weak $k$-flow conjecture combined with the existence of spanning trees with small degrees.

In Section 2 we collect the necessary tools for the proof of Theorem~\ref{thm-treedecomp}. The main idea of the proof is explained in Section 3, while the technical details can be found in Section 4. 
As an immediate consequence of Theorem~\ref{thm-treedecomp}, we get that Conjecture~\ref{conj-treedecomp} holds for trees of diameter at most 3. An explicit upper bound on the necessary edge-connectivity for these trees is derived in Section 5. 
In Section 6, we verify Conjecture~\ref{conj-treedecomp} for all trees of diameter 4. With a little more effort, the arguments in Section 6 can also be used to verify Conjecture~\ref{conj-treedecomp} for some trees of diameter 5, including the path of length 5. The details will appear in~\cite{thesis}. 
Finally, we show how Theorem~\ref{thm-treedecomp} extends to infinite graphs in Section 7.

\section{Methods}

Unless stated otherwise, all graphs in this paper are finite and loopless, but may have multiple edges. We write $V(G)$ and $E(G)$ for the vertex set and edge set of a graph $G$ respectively, and $e(G)$ for the number of edges of $G$. We denote the degree of a vertex $v$ in $G$ by $d(v,G)$, or by $d(v)$ if the graph is clear from the context. If the graph is directed, we denote the outdegree of a vertex $v$ by $d^+(v)$ and the indegree by $d^-(v)$. We write $P_k$ for the path on $k$ vertices, thus $P_k$ has length $k-1$.

An important tool for working with edge-connected graphs is the following reduction method due to Mader. Let $v$ be a vertex of even degree in a graph $G$. A \textit{lifting} at $v$ is the operation of deleting $v$ and adding a perfect matching between the neighbours of $v$ in $G$. We say that the lifting is \textit{connectivity-preserving}, if the edge-connectivity of the resulting graph is not smaller than the edge-connectivity of $G$. We shall use the following version of Mader's Theorem which was proved in~\cite{frank}.

\begin{theorem}\emph{\cite{frank, mader}}\label{mader}
Let $v$ be a vertex of even degree in a graph $G$. If $v$ is not incident with a cut-edge, then there exists a connectivity-preserving lifting at $v$. 
\end{theorem}

Another useful consequence of large edge-connectivity is the existence of edge-disjoint spanning trees. If a graph contains $k$ edge-disjoint spanning trees, then it is clearly $k$-edge-connected. Conversely, Nash-Williams~\cite{nash-williams} and Tutte~\cite{tutte} independently proved that large edge-connectivity implies the existence of many edge-disjoint spanning trees.

\begin{theorem}\emph{\cite{nash-williams, tutte}}\label{nash-williams}
Let $k$ be a natural number. If $G$ is a $2k$-edge-connected graph, then $G$ contains $k$ edge-disjoint spanning trees. 
\end{theorem}

Apart from many edge-disjoint spanning trees, large edge-connectivity also guarantees the existence of spanning trees with small vertex degrees. This has been investigated by several authors, see for example~\cite{strothmann}, \cite{ellingham}, \cite{hasanvand}, and~\cite{liu}. Small-degree spanning trees have already been used in~\cite{barat}, \cite{3paths}, \cite{4paths}, and \cite{2kpaths} to prove special cases of Conjecture~\ref{conj-treedecomp}. 

Our main interest is to prove the existence of $k(T)$. Since the method presented here will not result in the best possible upper bound on $k(T)$, we shall avoid the use of stronger but more technical statements for the sake of simplicity. The following theorem was proved in~\cite{ellingham} and is sufficient for our purposes.

\begin{theorem}\emph{\cite{ellingham}}\label{ellingham}
Let $k$ be a natural number. If $G$ is a $4k$-edge-connected graph, then $G$ has a spanning tree $T$ such that $d(v,T) < d(v,G)/k$ for every $v\in V(G)$.
\end{theorem}

Repeated application of Theorem~\ref{ellingham} also guarantees the existence of highly edge-connected subgraphs with small degrees.

\begin{lemma}\label{con+smalldeg}
Let $k$ and $q$ be natural numbers. If $G$ is a graph with $4kq$ edge-disjoint spanning trees, then $G$ has a spanning $q$-edge-connected subgraph $H$ such that $d(v,H) < d(v,G)/k$ for every $v\in V(G)$.
\end{lemma}

\begin{proof}
Let $G_i$ consist of $4k$ of the spanning trees for $i\in\{1,\ldots ,q\}$. By Theorem~\ref{ellingham}, for each $G_i$ we can find a spanning tree $T_i$ with $d(v,T_i) < d(v,G_i)/k$. Let $H$ be the union of $T_1,\ldots ,T_q$. Now $H$ is $q$-edge-connected and we have
\begin{align*}
d(v,H) = \sum_{i=1}^q d(v,T_i) < \frac{1}{k}\sum_{i=1}^q d(v,G_i) \leq \frac{1}{k}d(v,G).
\end{align*}
\end{proof}

Another important tool is a recent result on orientations modulo $k$. It was proved in~\cite{weak3flow} that every $(2k^2+k)$-edge-connected graph $G$ has an orientation such that every vertex gets a prescribed outdegree modulo $k$, provided that the sum of all prescribed outdegrees is congruent to $e(G)$ modulo $k$. In~\cite{weak3flowimproved}, the bound on the edge-connectivity has been improved to $3k-3$ for $k$ odd, and to $3k-2$ for $k$ even.

\begin{theorem}\emph{\cite{weak3flowimproved}}\label{weak-k-flow}
Let $G$ be a graph with $m$ edges, $k$ be a natural number, and $p:V(G)\to \Z /k\Z$ be a function satisfying $\sum_{v\in V(G)}p(v) \equiv m \mbox{ (mod k)}$. If $G$ is $(3k-2)$-edge-connected, then there exists an orientation of the edges of $G$ such that $d^+(v)\equiv p(v) \mbox{ (mod k)}$ for every $v\in V(G)$.
\end{theorem}

As an application of Theorem~\ref{weak-k-flow}, it was shown in~\cite{bistars} that a highly edge-connected bipartite graph $G$ with size divisible by $k$ can be decomposed into two $k$-edge-connected graphs $G_1$ and $G_2$ such that in $G_1$ all vertices of $A$ have degree divisible by $k$, and in $G_2$ all vertices of $B$ have degree divisible by $k$. Using the same proof, we can even achieve that both $G_1$ and $G_2$ have arbitrarily large edge-connectivity, which has also been used in~\cite{5path}. 

\begin{proposition}\emph{\cite{5path, bistars}}\label{prop-degdivbym}
Let $m$ and $l$ be natural numbers, $G$ be a bipartite graph on vertex classes $A_1$ and $A_2$, and suppose the size of $G$ is divisible by $m$. If $G$ has $3m-2+2l$ edge-disjoint spanning trees, then $G$ can be decomposed into two spanning $l$-edge-connected graphs $G_1$ and $G_2$ such that all vertices of $A_i$ have degree divisible by $m$ in $G_i$, for $i\in\{1,2\}$. 
\end{proposition}

\begin{proof}
Let $H_1$ and $H_2$ each be the union of $l$ of the spanning trees, and let $G'$ be the graph on the remaining edges. For $v$ in $A_i$ define $p(v)=m-d(v,H_i)$. Since $G'$ is $(3m-2)$-edge-connected, we can apply Theorem~\ref{weak-k-flow} to orient its edges so that each vertex $v$ has outdegree congruent to $p(v)$ modulo $m$. For $i\in\{1,2\}$, let $G_i$ be the union of $H_i$ and all edges oriented from $A_i$ to $A_{3-i}$.
\end{proof}

It was shown independently in~\cite{barat} and~\cite{bistars} that it is sufficient to prove Conjecture~\ref{conj-treedecomp} for bipartite graphs. The proofs there show that the following is true.

\begin{theorem}\emph{\cite{barat, bistars}}\label{bipartite}
Let $T$ be a tree on $m$ edges. If $G$ is a $(4k + 8m^{2m+3})$-edge-connected graph, then $G$ can be decomposed into a $k$-edge-connected bipartite graph~$G'$ and a graph $H$ admitting a $T$-decomposition.
\end{theorem}

In particular, it is sufficient to prove Theorem~\ref{thm-treedecomp} for bipartite graphs. Combined with Proposition~\ref{prop-degdivbym}, we may even assume that all vertices on one side of the partition have degree divisible by $m$.

\section{Proof strategy for Theorem~\ref{thm-treedecomp}}

The general idea of the proof of Theorem~\ref{thm-treedecomp} is the following. Given a tree $T$ with $m$ edges, let $T_A$ and $T_B$ be the two vertex classes induced by a proper 2-colouring of $V(T)$. We may assume that $T_B$ contains a leaf of $T$. We denote the non-leaves in $T_A$ by $t_1,\ldots ,t_a$ and the non-leaves in $T_B$ by $t_{a+1},\ldots ,t_{a+b}$. We colour the edges of $T$ with colours $1,\ldots ,m$ so that no two edges receive the same colour. For $i\in\{1,\ldots ,a+b\}$, we denote the set of colours at vertex $t_i$ by $T(i)$. Let $T(a+b+1)$ be the set of colours which are not contained in $T(i)$ for any $i\in\{a+1,\ldots ,a+b\}$. Notice that $T(a+b+1)$ is non-empty since $T_B$ contains a leaf.

To find a $T^*$-decomposition of a bipartite graph $G$, it is sufficient to find an edge-colouring of $G$ such that certain degree equations are satisfied. We write $d_i(v)$ to denote the degree of a vertex $v$ in colour $i$. For every $i\in\{1,\ldots ,a+b\}$, we consider a set of equations involving the colours in $T(i)$, namely $d_j(v)=d_k(v)$ for every $j,k\in T(i)$. However, we do not need these equations to be satisfied at every vertex of $G$. If $A$ and $B$ are the vertex classes of $G$, then we want the vertices in $A$ to satisfy the equations involving $T(1),\ldots ,T(a)$, and the vertices in $B$ to satisfy the equations involving $T(a+1),\ldots ,T(a+b)$. 

\begin{definition}
Let $G$ be a bipartite graph on vertex classes $A$ and $B$. We say an edge-colouring of $G$ with colours $1,\ldots ,m$ is \textit{$T$-equitable}, if 
\begin{itemize}
\item
for every $v\in A$, $i\in\{1,\ldots ,a\}$ and $j,k\in T(i)$ we have $d_j(v)=d_k(v)$, and
\item 
for every $v\in B$, $i\in\{a+1,\ldots ,a+b\}$ and $j,k\in T(i)$ we have $d_j(v)=d_k(v)$.
\end{itemize}
\end{definition}

\begin{lemma}\label{equivtreedecomp}
Let $G$ be a bipartite graph on vertex classes $A$ and $B$. If $G$ admits a $T$-equitable colouring then $G$ has a $T^*$-decomposition.
\end{lemma}  
\begin{proof}
We prove by induction on $m$ that we can decompose $G$ into coloured homomorphic copies of $T$ so that, for every homomorphic copy of $T$, the vertices corresponding to $T_A$ lie in $A$. There is nothing to show for $m=1$, so we may assume $m\geq 2$ and that the statement is true for all smaller $m$. We may also assume that the edge of $T$ coloured $m$ is incident to $t_1$ in $T_A$ and to a leaf in $T_B$. Let $T'$ be the tree we get by deleting this edge, and let $G'$ be the graph we get by deleting all edges with colour $m$ in $G$. Now $G'$ satisfies the corresponding equations for $T'$, so we can find a decomposition into coloured homomorphic copies of $T'$ that are all oriented the same way. Since every vertex $v$ of $G$ in $A$ satisfies $d_j(v)=d_k(v)$ for $j,k\in T(1)$, the number of copies of $T'$ where $v$ is the image of $t_1$ is the same as the number of edges coloured $m$ at $v$. Thus we can extend every homomorphic copy of $T'$ to a homomorphic copy of $T$, resulting in a $T^*$-decomposition.
\end{proof}

As noted at the end of Section 2, it suffices to prove Theorem~\ref{thm-treedecomp} for bipartite graphs~$G$ with vertex classes $A$ and $B$ where all vertices in $A$ have degree divisible by $m$. In this situation we can construct an edge-colouring where the vertices in $A$ satisfy an even stronger condition than required in the lemma above, namely that they have the same degree in each colour. In other words, we are going to construct an edge-colouring with colours $1,\ldots ,m$ such that
\begin{align}\label{degequal}
d_i(v)=\frac{1}{m}d(v)
\end{align}
for $v\in A$, $i\in\{1,\ldots ,m\}$, while the vertices in $B$ satisfy the same equations as in Lemma~\ref{equivtreedecomp}. This means that in the corresponding $T^*$-decomposition, for every $x,y\in T_A$ and $v\in A$, the number of homomorphic copies where $v$ is the image of $x$ is the same as the number of homomorphic copies where $v$ is the image of $y$.

The existence of such an edge-colouring is an easy consequence of the following Theorem, which is proved in Section 4.

\begin{theorem}\label{thm-technical}
For all natural numbers $m$ and $\lambda$, there exists a natural number $f(m,\lambda)$ such that the following holds:

If $m_1,\ldots, m_{b+1}$ are positive integers satisfying $m=m_1 + \ldots + m_{b+1}$, and if $G$ is a $f(m,\lambda)$-edge-connected bipartite graph on vertex classes $A$ and $B$ in which all vertices in $A$ have degree divisible by $m$, then we can decompose $G$ into $b+1$ spanning $\lambda$-edge-connected graphs $G_1,\ldots, G_{b+1}$ such that
\begin{itemize}
\item
$d(v,G_i)=\frac{m_i}{m}d(v,G)$ for $v\in A$ and $i\in\{1,\ldots ,b+1\}$, and
\item
$d(v,G_i)$ is divisible by $m_i$ for $v\in B$ and $i\in\{1,\ldots ,b\}$.
\end{itemize}
\end{theorem}

Notice that it is not possible to achieve that also $d(v,G_{b+1})$ is divisible by $m_{b+1}$ for $v\in B$, since for example all of $m_1,\ldots ,m_{b+1}$ could be even, but $B$ could have vertices of odd degree.

Using Theorem~\ref{thm-technical}, we can easily construct a $T$-equitable edge-colouring of $G$.
\begin{theorem}\label{thm-T-equitable}
Let $G$ be a bipartite graph on vertex classes $A$ and $B$ in which all vertices in $A$ have degree divisible by $m$. If $G$ is $f(m,md)$-edge-connected, where $f$ denotes the function defined by Theorem~\ref{thm-technical}, then $G$ admits a $T$-equitable edge-colouring such that the minimum degree in each colour is at least $d$.
\end{theorem}
\begin{proof}
For $i\in\{1,\ldots ,b+1\}$, let $m_i=|T(a+i)|$.  Notice that every colour appears in precisely one of $T(a+1),\ldots , T(a+b+1)$, so we have $m=m_1 + \ldots + m_{b+1}$. Thus, we can apply Theorem~\ref{thm-technical} to get a decomposition of $G$ into $md$-edge-connected graphs $G_1,\ldots ,G_{b+1}$ such that $d(v,G_i)=\frac{m_i}{m}d(v,G)$ for $v\in A$, $i\in\{1,\ldots ,b+1\}$, and $d(v,G_i)$ is divisible by $m_i$ for $v\in B$, $i\in\{1,\ldots ,b\}$.

For $i\in\{1,\ldots ,b\}$, every vertex of $G_i$ has degree divisible by $m_i$, so we can split vertices to get an $m_i$-regular graph $G_i'$. 
We also split each vertex in $G_{b+1}$ into vertices of degree $m_{b+1}$ and possibly one vertex of degree less than $m_{b+1}$, resulting in a graph $G_{b+1}'$.
A well-known result by K\"onig states that every $k$-regular bipartite graph has a proper edge-colouring with $k$ colours, see for example Proposition 5.3.1. in~\cite{diestel}.
Thus, there exists a proper edge-colouring of $G_i'$ with the $m_i$ colours in $T(a+i)$ for every $i\in \{1,\ldots ,b+1\}$. 
This corresponds to an edge-colouring of $G_i$ such that
$$d_j(v,G_i) = \frac{1}{m_i}d(v,G_i) = \frac{1}{m}d(v,G) $$
for $j\in T(a+i)$ and $v\in A$. By construction, we also have $d_j(v)=d_k(v)$ for $v\in B$, $i\in\{a+1,\ldots, a+b\}$, $j,k\in T(i)$, so the colouring is $T$-equitable. Since the minimum degree of $G_i$ is at least $md$, the minimum degree in each color in $G$ is at least $d$. 
\end{proof}
Now the main result of this paper follows immediately.
\begin{proof}[Proof of Theorem~\ref{thm-treedecomp}]
We may assume by Proposition~\ref{prop-degdivbym} and Theorem~\ref{bipartite} that $G$ is a bipartite graph on vertex classes $A$ and $B$, and that every vertex in $A$ has degree divisible by $m$. If $G$ is $f(m,1)$-edge-connected, where $f$ denotes the function defined by Theorem~\ref{thm-technical}, then $G$ admits a $T$-equitable colouring by Theorem~\ref{thm-T-equitable}. Thus, $G$ has a $T^*$-decomposition by Lemma~\ref{equivtreedecomp}.  
\end{proof}

\section{Proof of Theorem~\ref{thm-technical}}

The following lemma is an easy application of Theorem~\ref{mader} and Theorem~\ref{weak-k-flow}. It is a slight generalization of an argument that was already used in~\cite{bistars} to prove Conjecture~\ref{conj-treedecomp} for a class of bistars.

\begin{lemma}\label{lemma-1/2}  
Let $G$ be a $(3k-2)$-edge-connected bipartite graph on classes $A$ and $B$, where each vertex in $A$ has even degree. For every function $p:B\to \Z /k\Z$ satisfying
\begin{align*}
\sum_{v\in B} p(v) \equiv \frac{e(G)}{2} \mbox{\hspace{10mm}(mod k)}\,,
\end{align*}
there exists a subgraph $H$ of $G$ with 

$d(v,H)=\frac{1}{2}d(v,G)$ for $v\in A$, and 

$d(v, H)\equiv p(v) \mbox{ (mod k)}$ for $v\in B$.
\end{lemma}

\begin{proof}
By Theorem~\ref{mader}, we can lift each vertex in $A$ so that the resulting graph $G'$ is still $(3k-2)$-edge-connected. By Theorem~\ref{weak-k-flow}, we can orient the edges of $G'$ such that each vertex $v$ has outdegree congruent to $p(v)$ modulo $k$. We can also orient the edges of $G$ such that every directed edge of $G'$ corresponds to a directed path of length 2 in $G$. This yields an orientation of $G$ where each vertex $v$ in $B$ has an outdegree congruent to $p(v)$ modulo $k$, and each vertex in $A$ has the same out- and indegree. Now the subgraph consisting of the edges oriented from $B$ to $A$ is as required.
\end{proof}

The case where we want the subgraph $H$ to contain only $1/m$ of the edges at every vertex in $A$, for some $m\geq 3$, can easily be reduced to the case $m=2$.

\begin{proposition}\label{prop-1/m}
Let $m$ and $k$ be natural numbers with $m\geq 2$, and let $G$ be a bipartite graph on classes $A$ and $B$ with $12km$ edge-disjoint spanning trees, where each vertex in $A$ has degree divisible by $m$. For every function $p:B\to \Z /k\Z$ satisfying
\begin{align*}
\sum_{v\in B} p(v) \equiv \frac{e(G)}{m} \mbox{\hspace{10mm}(mod k)}\,,
\end{align*}
there is a subgraph $H$ of $G$ with 

$d(v,H)=\frac{1}{m}d(v,G)$ for $v\in A$ and 

$d(v, H)\equiv p(v) \mbox{ (mod k)}$ for $v\in B$.
\end{proposition}

\begin{proof} 
By Lemma~\ref{con+smalldeg}, we can find a spanning $3k$-edge-connected subgraph $G'$ with $d(v,G') < \frac{1}{m}d(v,G)$. 
We add some edges to get a graph $G''$ in which all vertices in $A$ have degree exactly $\frac{2}{m}d(v,G)$. Now we can use Lemma~\ref{lemma-1/2} to find a subgraph $H$  of~$G''$ with $d(v,H) = \frac{1}{2}d(v,G'') = \frac{1}{m}d(v,G)$ for $v\in A$, and $d(v,H) \equiv p(v)$ modulo $k$ for $v\in B$.
\end{proof}

To get a decomposition into several graphs as in Theorem~\ref{thm-technical}, we want to use induction. To do so, we need that $G - E(H)$ still has large edge-connectivity. The following lemma shows that this can be achieved by increasing the edge-connectivity of~$G$.

\begin{lemma}\label{lemma-1/m+lambda}
Let $k$, $m$ and $\lambda$ be natural numbers with $m\geq 2$. Let $G$ be a bipartite graph on classes $A$ and $B$ with $8\lambda m^2 + 12km$ edge-disjoint spanning trees, where each vertex in $A$ has degree divisible by $m$. For every function $p:B\to \Z /k\Z $ satisfying
\begin{align*}
\sum_{v\in B} p(v) \equiv \frac{e(G)}{m} \mbox{\hspace{10mm}(mod k)}\,,
\end{align*}
there is a decomposition of $G$ into $\lambda$-edge-connected subgraphs $G_1$ and $G_2$ with 

$d(v,G_1)=\frac{1}{m}d(v,G)$ for $v\in A$ and 

$d(v,G_1)\equiv p(v) \mbox{(mod k)}$ for $v\in B$.
\end{lemma}

\begin{proof}
Let $H_1$ and $H_2$ each be the union of $4\lambda m^2$ of the spanning trees, and let $H_3$ be the union of the remaining edges. By Lemma~\ref{con+smalldeg}, we can find a spanning $\lambda$-edge-connected subgraph $H_i'$ of $H_i$ satisfying 
\begin{align}\label{eq-Hi'}
d(v,H_i') < \frac{1}{m^2}d(v,H_i) < \frac{1}{m^2}d(v,G) 
\end{align}
for $i\in\{1,2\}$, and a spanning $3k$-edge-connected subgraph $H_3'$ of $H_3$ satisfying 
\begin{align}\label{eq-H3'}
d(v,H_3') < \frac{1}{m}d(v,H_3) < \frac{1}{m}d(v,G)\,.
\end{align} 
We are going to colour the edges of $G$ with colours 1 and 2 so that for $i\in\{1,2\}$ the graph $G_i$ induced by the edges coloured $i$ will be as required. As before, we denote the degree of a vertex $v$ in colour $i$ by $d_i(v)$.

We start by colouring all edges in $H_1'$ with colour 1, and all edges in $H_2'$ with colour 2. This ensures that both $G_1$ and $G_2$ will be $\lambda$-edge-connected. We also want 
\begin{align}\label{eq-m-1}
(m-1)d_1(v)=d_2(v)
\end{align}
to hold for $v\in A$. For every vertex in $A$, we colour more of its edges with colours 1 or 2 so that~\eqref{eq-m-1} is satisfied. We do it in such a way that the number of edges we colour is minimal. Because of~\eqref{eq-Hi'}, we have coloured at most $\frac{1}{m}d(v,G)$ edges incident with $v$ for every $v\in A$. Because of~\eqref{eq-H3'}, we can assume that all these coloured edges are outside of $H_3'$. Let $G'$ be the graph consisting of all edges we have coloured so far, and let $G''$ be the graph induced by the remaining edges. In particular, $G'$ satisfies~\eqref{eq-m-1} and $G''$ contains $H_3'$. In $G'$ every vertex in $A$ has degree divisible by $m$, so this must also be the case in $G''$. Since $d(v,G')\leq \frac{1}{m}d(v,G)$ for $v\in A$, we have $d(v,G'')\geq \frac{m-1}{m}d(v,G) \geq \frac{1}{2}d(v,G)$ and thus also
\begin{align*}
d(v,H_3') < \frac{1}{m}d(v,G) \leq \frac{2}{m}d(v,G'')
\end{align*}
for every $v\in A$. Now we repeat the argument from the proof of Proposition~\ref{prop-1/m}: We find a subgraph $G'''$ of $G''$ containing $H_3'$ and satisfying $d(v,G''') = \frac{2}{m}d(v,G'')$ for $v\in A$. Let $p':B\to \Z /k\Z$ be the function defined by $p'(v)=p(v)-d_1(v,G')$ for $v\in B$. By Lemma~\ref{lemma-1/2}, we can find a subgraph $H$ of $G'''$ satisfying
\begin{align*}
d(v,H) = \frac{1}{2}d(v,G''') = \frac{1}{m}d(v,G'')
\end{align*}
for $v\in A$, and $d(v,H) \equiv p'(v)$ modulo $k$ for $v\in B$. Colouring the edges of $H$ with colour 1 and the remaining edges of $G''$ with colour 2 yields a decomposition as desired.
\end{proof}

Repeated application of Lemma~\ref{lemma-1/m+lambda} results in the following proposition.

\begin{proposition}\label{prop-technical}
For all natural numbers $k, m,$ and $\lambda$, there exists a natural number $f_k(m,\lambda)$ such that the following holds:

If $G$ is a $f_k(m,\lambda)$-edge-connected bipartite graph on vertex classes $A$ and $B$, in which all vertices in $A$ have degree divisible by $m$, and $p_1,\ldots ,p_{m-1}$ are functions $p_i:B\to \Z /k\Z$ satisfying 
$$\sum_{v\in B} p_i(v) \equiv \frac{e(G)}{m} \hspace{2mm}(\mbox{mod } k)$$ 
for $i\in\{1,\ldots ,m-1\}$, then there is a decomposition of $G$ into $m$ spanning $\lambda$-edge-connected graphs $G_1,\ldots, G_m$ such that

$d(v,G_i)=\frac{1}{m}d(v,G)$ for $v\in A$ and $i\in\{1,\ldots ,m\}$, and
 
$d(v,G_i)\equiv p_i(v) \hspace{2mm}(\mbox{mod } k)$ for $v\in B$ and $i\in\{1,\ldots ,m-1\}$.
\end{proposition}

\begin{proof} 
We use induction on $m$. By Theorem~\ref{nash-williams} and Lemma~\ref{lemma-1/m+lambda}, the statement is true for $m=2$ and $f_k(2,\lambda) = 64\lambda + 48k$. Thus, we may assume $m\geq 3$ and $f_k(m-1,\lambda)$ exists. Set
\begin{align*}
f_k(m,\lambda) = 16f_k(m-1,\lambda)m^2 + 24km\,.
\end{align*} 
If $G$ is $f_k(m,\lambda)$-edge-connected, then we can use Lemma~\ref{lemma-1/m+lambda} to decompose $G$ into $f_k(m-1,\lambda)$-edge-connected subgraphs $G'$ and $G_{m-1}$ such that $d(v,G_{m-1})=d(v,G)/m$ for $v$ in $A$ and $d(v,G_{m-1})\equiv p_{m-1}(v)$ modulo $k$ for $v$ in $B$. Now we can use the induction hypothesis for $m-1$ with functions $p_1,\ldots ,p_{m-2}$ to decompose $G'$ into $m-1$ spanning $\lambda$-edge-connected subgraphs $G_1,\ldots ,G_{m-2}, G_m$ satisfying the conditions above. These graphs together with $G_{m-1}$ decompose $G$ as desired.
\end{proof}

Now Theorem~\ref{thm-technical} follows easily.

\begin{proof}[Proof of Theorem~\ref{thm-technical}]

For a partition $P$ of $m$ into parts $m_1,\ldots ,m_{b+1}$, we define $\pi (P)$ to be the product of $m_1,\ldots ,m_{b+1}$. We are going to show that every $f_{\pi (P)}(m,\lambda)$-edge-connected graph has a decomposition satisfying the conditions in the conclusion of Theorem~\ref{thm-technical}, where $f_{\pi (P)}$ is the function defined by Proposition~\ref{prop-technical}. Since there are only finitely many partitions of $m$ into positive integers, we can then choose $f(m,\lambda)$ as the maximum of all values $f_{\pi (P)}(m,\lambda)$ over all partitions $P$ of $m$.

Let $m=m_1 + \ldots + m_{b+1}$ be a partition of $m$ into positive integers, and let $k$ be the product of $m_1,\ldots ,m_{b+1}$. Let $G$ be $f_k(m,\lambda)$-edge-connected. We pick some function $q:B\to \Z /k\Z$ satisfying
$$\sum_{v\in B} q(v) \equiv \frac{e(G)}{m} \hspace{2mm}(\mbox{mod } k)\,,$$
and we apply Proposition~\ref{prop-technical} with $p_1=\ldots =p_{m-1}=q$ to get $\lambda$-edge-connected graphs $H_1,\ldots ,H_m$ satisfying
$d(v,H_i)=\frac{1}{m}d(v,G)$ for $v\in A$, $i\in\{1,\ldots ,m\}$, and
$d(v,H_i)\equiv q(v) \hspace{2mm}(\mbox{mod } k)$ for $v\in B$, $i\in\{1,\ldots ,m-1\}$. We construct graphs $G_1,\ldots ,G_{b+1}$ such that $G_i$ is the union of precisely $m_i$ of the graphs $H_j$, every $H_j$ is contained in precisely one of the $G_i$, and $G_{b+1}$ contains $H_m$. 
Now we have
$$d(v,G_i)=\frac{m_i}{m}d(v,G)$$
for $v\in A$, $i\in\{1,\ldots ,b+1\}$, and 
$$d(v,G_i)\equiv m_iq(v) \hspace{2mm}(\mbox{mod } k)$$
for $v\in B$, $i\in\{1,\ldots ,b\}$. Since $m_i$ divides $k$ for $i\in\{1,\ldots ,b\}$, we have that $d(v,G_i)$ is divisible by $m_i$ for $v\in B$, $i\in\{1,\ldots ,b\}$, so the graphs $G_1,\ldots ,G_{b+1}$ are as desired.
\end{proof}

\section{Trees of diameter 3}
Let $T$ be a tree of diameter 3. As before, it is sufficient to consider the case where the simple graph $G$ we want to decompose is bipartite. In particular, we can assume that the girth of $G$ is at least 4. Thus, every $T^*$-decomposition of $G$ is also a $T$-decomposition of $G$, and so Theorem~\ref{thm-treedecomp} immediately implies Conjecture~\ref{conj-treedecomp} in this case. In the following we are going to take a closer look at the value $k(T)$ resulting from the proof of Theorem~\ref{thm-treedecomp}.

Let $S(k,l)$ denote the bistar with two adjacent vertices of degree $k$ and $l$ respectively, and all other vertices having degree 1. Every tree of diameter 3 is isomorphic to a bistar $S(k,l)$ for some natural numbers $k$ and $l$ with $1<k\leq l$. 

The following proposition is a special case of Theorem~\ref{thm-technical}. We give a direct proof here to get a better edge-connectivity. For this we use the following strengthening of Theorem~\ref{ellingham}, which is Corollary 21 in~\cite{ellingham}: 
For every $\varepsilon$ with $0<\varepsilon < 1$, if $G$ is $\lceil \frac{4}{\varepsilon} \rceil$-edge-connected, then $G$ has a spanning tree $T$ such that $d(v,T) < \varepsilon \,d(v,G)$ for every $v \in V(G)$. This results in a canonical strengthening of Lemma~\ref{con+smalldeg}, which we shall use in the following proof.

\begin{proposition}\label{prop-bistar}
Let $k$, $l$ be natural numbers with $1<k\leq l$, and let $m=k+l-1$. Assume $G$ is a bipartite graph on vertex classes $A$ and $B$ where all vertices in $A$ have degree divisible by $m$. If $G$ is $3l\lceil \frac{2m}{k-1}\rceil$-edge-connected, then $G$ has a decomposition into two graphs $G_1$ and $G_2$ such that
\begin{itemize}
\item
$d(v,G_1)=\frac{k-1}{m}d(v,G)$ for $v\in A$, and
\item
$d(v,G_2)$ is divisible by $l$ for $v\in B$.
\end{itemize}
\end{proposition}
\begin{proof}
The proof is essentially the same as the proof of Proposition~\ref{prop-1/m}. By the strengthening of Lemma~\ref{con+smalldeg}, we can find a spanning $3l$-edge-connected subgraph $G'$ with $d(v,G') < \frac{2(k-1)}{m}d(v,G)$. 
Since $2(k-1)<m$, we can add some edges of $G$ to $G'$ to get a graph $G''$ in which all vertices in $A$ have degree precisely $\frac{2(k-1)}{m}d(v,G)$. Now we use Lemma~\ref{lemma-1/2} with the function $p:B\to \Z /l\Z$ defined by $p(v)\equiv d(v,G)$ modulo $l$. The resulting subgraph $G_1$ of $G''$ satisfies $d(v,G_1) = \frac{1}{2}d(v,G'') = \frac{k-1}{m}d(v,G)$ for $v$ in $A$, and $d(v,G_1) \equiv p(v)$ modulo $l$ for $v$ in $B$. Let $G_2$ denote the graph $G-E(G_1)$, then $d(v,G_2)=d(v,G)-d(v,G_1)\equiv 0$ modulo $l$ for $v\in B$, so the graphs $G_1$ and $G_2$ are as desired.
\end{proof}

Given a decomposition of a graph $G$ into graphs $G_1$ and $G_2$ as above, we immediately get an $S(k,l)$-decomposition by the same arguments as in Section 3: We edge-colour $G_2$ with $l$ colours so that every vertex has the same degree in each colour, and we edge-colour $G_1$ with $k-1$ different colours so that every vertex in $A$ has the same degree in all $k+l-1$ colours. Now we get an $S(k,l)$-decomposition by Lemma~\ref{equivtreedecomp}, where the vertices of degree $k$ lie in $A$ and the vertices of degree $l$ lie in $B$.

It was proved in~\cite{bistars} that every $180k^4$-edge-connected bipartite simple graph with size divisible by $2k$ has an $S(k,k+1)$-decomposition. Combining Proposition~\ref{prop-bistar} with Proposition~\ref{prop-degdivbym}, we get the following stronger result.

\begin{theorem}\label{thm-bistar-bipartite}
Let $k$, $l$ be natural numbers with $1<k\leq l$, and let $m=k+l-1$. Every $(12l\lceil \frac{2m}{k-1}\rceil+6m-4)$-edge-connected bipartite graph with size divisible by $m$ has an $S(k,l)$-decomposition.

In particular, every $(72k+236)$-edge-connected bipartite simple graph with size divisible by $2k$ has an $S(k,k+1)$-decomposition.
\end{theorem} 

For trees of diameter 3 the general reduction to the bipartite case from Proposition~\ref{bipartite} was done more efficiently in~\cite{bistars}.

\begin{theorem}\emph{\cite{bistars}}\label{bipartite-diam3}
Let $T$ be a tree on $m$ edges with diameter 3, and let $k$ be a natural number. If $G$ is a $(4k + 16m(m+1))$-edge-connected graph, then $G$ can be decomposed into a $k$-edge-connected bipartite graph~$G'$ and a graph $H$ admitting a $T$-decomposition.
\end{theorem}

It was proved in~\cite{bistars} that every $784k^4$-edge-connected simple graph with size divisible by $2k$ has a $S(k,k+1)$-decomposition. Combining Theorem~\ref{thm-bistar-bipartite} with Theorem~\ref{bipartite-diam3}, we get the following more general result.

\begin{theorem} 
Let $k$, $l$ be natural numbers with $1<k\leq l$, and let $m=k+l-1$. Every $112m^2$-edge-connected graph of size divisible by $m$ has an $S(k,l)$-decomposition.
\end{theorem} 

For $k=l=2$, the bistar $S(k,l)$ is a path of length 3. This special case was investigated in~\cite{3paths}, where it was shown that every 171-edge-connected simple graph with size divisible by 3 admits a $P_4$-decomposition. In the proof it was shown that every 2-edge-connected bipartite simple graph where all vertices on one side have degree divisible by 3 admits a decomposition into paths of length 3. Note that for $m$ odd $3m-3+2l$ edge-disjoint spanning trees suffice in Proposition~\ref{prop-degdivbym}, so every bipartite simple graph with $10$ edge-disjoint spanning trees has a $P_4$-decomposition. Replacing this part in the proof of~\cite{3paths}, we get that every 63-edge-connected simple graph with size divisible by 3 can be decomposed into paths of length 3.

\section{Trees of diameter 4}

Let $T$ be a tree of diameter 4. We may assume that the graph $G$ we want to decompose is bipartite and thus has girth 4, so the only problem is that some homomorphic copies in the $T^*$-decomposition might contain 4-cycles. To take care of this, we shall start with a $T^*$-decomposition and try to improve it by switching leaf edges between different homomorphic copies. It is essential that Theorem~\ref{thm-technical} results in a decomposition where every vertex has a large degree in every colour, so that we have enough freedom to make switches. This method can be used whenever the girth of $G$ is at least the diameter of $T$. Before we see how this strategy works in a general setting, we investigate the path of length 4. Notice that the minimum degree condition in the next proposition cannot be omitted, since a cycle of length 4 satisfies all other conditions.

\begin{proposition}\label{prop-4path}
Let $G$ be a bipartite simple graph on vertex classes $A$ and $B$ with size divisible by 4, where the vertices in $A$ have even degree. 

If $G$ is 2-edge-connected, then $G$ has a $P_5^*$-decomposition. 

If $G$ is 2-edge-connected and the vertices in $A$ have minimum degree 4, then $G$ has a $P_5$-decomposition.
\end{proposition}
\begin{proof}
We can lift the vertices in $A$ such that the resulting graph $G'$ is still 2-edge-connected. Since $G'$ is connected and has an even number of edges, it is possible to orient its edges so that every vertex has an even outdegree. Every directed edge in $G'$ corresponds to a directed path of length 2 in $G$. We colour the first edge of each of these directed paths red and the second edge blue. Now every vertex in $A$ has the same degree in red and blue, and the vertices in $B$ have even degree in red.

We first match the red edges at every vertex in $B$ arbitrarily, these will be the two middle edges of the paths of length 4. For each red path of length 2, we need to add a blue edge to each of its ends. Since the vertices in $A$ have the same degree in red and blue, we can find a matching of the blue edges and the ends of the red paths resulting in a $P_5^*$-decomposition. This proves the first part of the proposition, so we may now assume that the vertices in $A$ have minimum degree $2d$ for some $d\geq 2$.

Let $x$ be a vertex in $B$. We say that a homomorphic copy of $P_5$ has a \textit{conflict} at $x$, if $x$ is incident to both blue edges of that copy. We choose a matching of the red and blue edges such that the number of conflicts, and thus the number of 4-cycles, is minimal. 

Suppose there is a conflict at some vertex $x$ in $B$. Consider the directed graph $D(x)$ where the vertices are the homomorphic copies of $P_5$ in our decomposition. For two homomorphic copies $T_1$ and $T_2$, we add an edge oriented from $T_1$ to $T_2$ in $D(x)$ for every $a\in A$ such that $ax$ is a blue edge of $T_1$, and there is a vertex $b\in B$ for which $ab$ is a blue edge of $T_2$. The idea is that it is then possible to switch the blue edge $ax$ of $T_1$ with the blue edge $ab$ of $T_2$ to decrease the number of occurences of $x$ in $T_1$. Notice that such a switch might create a new conflict at $x$, but not at any other vertex.

In $D(x)$, each vertex has either outdegree 0 (if $x$ is not a leaf in the homomorphic copy), or it has at least outdegree $d-1$. Notice that every vertex with positive outdegree has indegree at most 1, since the corresponding homomorphic copy has at most one blue edge not incident with $x$, say $ab$ with $a\in A$, $b\in B$, and there is at most one homomorphic copy in which $ax$ is a blue edge.

Since we assumed there is a conflict at $x$, there is a vertex $v$ in $D(x)$ with outdegree at least $2(d-1)$ and indegree 0. Let $X$ be the set of vertices we can reach from $v$ via a directed path, including $v$. Suppose every vertex in $X$ has positive outdegree, then the subgraph induced by $X$ contains at least $(|X|+1)(d-1)$ edges. However, it can contain at most $|X|-1$ edges, since every vertex has indegree at most 1, and $v$ has indegree 0. Thus, there is a directed path in $D(x)$ from $v$ to a vertex of outdegree 0, and making the switches corresponding to the edges on this path reduces the number of conflicts by 1, contradicting our assumption.
\end{proof}

It was shown in~\cite{4paths} that every $10^{10^{10^{14}}}$-edge-connected graph of size divisible by 4 has a decomposition into paths of length 4. Using the proposition above, this bound on the edge-connectivity can be significantly improved. By Proposition~\ref{prop-degdivbym}, every bipartite simple graph with 14 spanning trees and size divisible by 4 can be decomposed into two graphs satisfying the conditions of Proposition~\ref{prop-4path}. Combining this with the first part of the proof in~\cite{4paths} shows that every 107-edge-connected graph of size divisible by 4 has a $P_5$-decomposition.

Conjecture~\ref{conj-treedecomp} can be restated as follows:
For every tree $T$ on $m$ edges, and every natural number $g$ with $g\geq 3$, there exists a natural number $k(T,g)$ such that every $k(T,g)$-edge-connected simple graph with girth at least $g$ and size divisible by $m$ has a $T$-decomposition. The existence of $k(T,4)$ is equivalent to the existence of $k(T)$ by Theorem~\ref{bipartite}. 
If $d$ denotes the diameter of $T$, then we know by Theorem~\ref{thm-treedecomp} that $k(T,d+1)$ exists. The following theorem shows that also $k(T,d)$ exists. This implies Conjecture~\ref{conj-treedecomp} for trees of diameter 4.

\begin{theorem}
Let $T$ be a tree of size $m$ and diameter $d$. There exists a natural number $k(T,d)$ such that every $k(T,d)$-edge-connected simple graph $G$ with girth at least $d$ and size divisible by $m$ has a $T$-decomposition.
\end{theorem}

\begin{proof}
We may assume as usual that $G$ is bipartite on vertex classes $A$ and $B$ and that all vertices in $A$ have degree divisible by $m$. We are going to show that it suffices for $G$ to be $f(m,2m)$-edge-connected, where $f$ is the function defined by Theorem~\ref{thm-technical}.

If $d$ is odd, then $G$ has girth at least $d+1$, so the conclusion follows from Theorem~\ref{thm-treedecomp}. Thus, we may assume that $d$ is even. Let $T_A$ and $T_B$ be the two vertex classes defined by a proper 2-colouring of $T$. We may assume that $T_B$ contains the ends of every longest path in $T$, since $d$ is even. We colour the edges of $T$ that are incident with leaves in $T_B$ blue, and the remaining edges red.

Let $\lambda$ be a natural number with $\lambda \geq 2m$, and assume $G$ is $f(m,\lambda)$-edge-connected. We can use Theorem~\ref{thm-technical} and Lemma~\ref{equivtreedecomp} to get a $T^*$-decomposition, where all vertices in $T_A$ correspond to vertices in $A$ in the homomorphic copies. We colour the edges of $G$ red and blue according to the colour of the edge they correspond to in $T$. Notice that by the proof of Theorem~\ref{thm-treedecomp}, the subgraph $G_{b+1}$ in Theorem~\ref{thm-technical} corresponds precisely to the edges coloured blue in $G$, so every vertex in $A$ is incident to at least $\lambda$ blue edges.

Since $G$ has girth $d$, the only way a homomorphic copy can fail to be an isomorphic copy of $T$ is if it contains a cycle of length $d$ or, equivalently, two blue edges intersecting at a vertex in $B$. As in the previous proof, we shall repair this by switching one of the blue edges with a blue edge from another homomorphic copy. We are not going to make any changes to the red edges, every red part of a homomorphic copy in the $T^*$-decomposition will be the red part of an isomorphic copy in the $T$-decomposition.

For $x\in B$, a \textit{conflict} at $x$ is a pair of blue edges contained in the same homomorphic copy of $T$ such that both of them are incident with $x$. Notice that one homomorphic copy may have several conflicts at $x$. Out of all $T^*$-decompositions we can get by switching blue edges between copies of our original $T^*$-decomposition, we choose one for which the number of conflicts is minimal.

Suppose there is a conflict at some vertex $x\in B$. Consider the directed graph $D(x)$ where the vertices are the homomorphic copies of $T$ in the $T^*$-decomposition. For two homomorphic copies $T_1$ and $T_2$, we add an edge oriented from $T_1$ to $T_2$ in $D(x)$ for every $a\in A$, $b\in B-V(T_1)$ such that $ax$ is a blue edge of $T_1$ and $ab$ is a blue edge of $T_2$. 
Again, the idea is to switch the blue edge $ax$ of $T_1$ with the blue edge $ab$ of $T_2$ to decrease the number of occurences of $x$ in $T_1$. Notice that such a switch might create a new conflict at $x$, but since $b$ is not contained in $T_1$ it will not create a conflict at any other vertex. Since less than $m$ of the blue edges at $a$ are incident with another vertex of $T_1$, there are at least $\lambda -m$ blue edges we can choose for the switch. In particular, every vertex of positive outdegree in $D(x)$ has outdegree at least $\lambda - m$.

Let $v$ be a vertex of $D(x)$ corresponding to a homomorphic copy containing a conflict at $x$, so $v$ has outdegree at least $2(\lambda -m)$. Let $X$ denote the set of vertices in $D(x)$ we can reach from $v$ via a directed path, including $v$. If every vertex in $X$ has positive outdegree, then the subgraph induced by $X$ has more than $(\lambda -m)|X|$ edges. However, every vertex of $D(x)$ has indegree at most $b$, where $b$ denotes the number of blue edges of $T$. Thus, the graph induced by $|X|$ has less than $b|X|$ edges, which is at most $(\lambda - m)|X|$ for $\lambda \geq 2m$. This shows that there must be a vertex $u$ of outdegree 0 in $X$. Now making the switches corresponding to the edges of the directed path from $v$ to $u$ results in a $T^*$-decomposition with fewer conflicts, contradicting our assumption.
\end{proof}

\begin{corollary}
For every tree $T$ with size $m$ and diameter at most 4, there exists a natural number $k(T)$ such that every $k(T)$-edge-connected graph with size divisible by $m$ has a $T$-decomposition.
\end{corollary}

\section{Extension to infinite graphs}

It is natural to ask whether Conjecture~\ref{conj-treedecomp} could hold for infinite graphs. We conjecture the following canonical extension.

\begin{conjecture}\label{conj-infinite}
For every tree $T$, there exists a natural number $k_{\infty} (T)$ such that every $k_{\infty} (T)$-edge-connected infinite graph $G$ has a $T$-decomposition.
\end{conjecture}

Apart from the locally finite case, it is not known whether Conjecture~\ref{conj-treedecomp} implies Conjecture~\ref{conj-infinite}. However, the situation changes if we consider $T^*$-decompositions instead.  

\begin{theorem}\emph{\cite{thesis}}\label{thm-inf}
Let $T$ be a tree of size $m$. Every $(k_h(T)+m^2-m)$-edge-connected infinite graph has a $T^*$-decomposition.
\end{theorem}

Notice that $k_h(T)$-edge-connectivity is not sufficient: Every connected graph of even size decomposes into paths of length 2, but the infinite star where all edges apart from one are subdivided has no such decomposition.

We shall just sketch the proof of Theorem~\ref{thm-inf} here, the details can be found in~\cite{thesis}.

\begin{proof}[Sketch of Proof]
By a standard argument using K\"onig's Infinity Lemma, every $(k_h(T)+m-1)$-edge-connected locally finite graph has a $T^*$-decomposition. By a result in~\cite{laviolette}, every infinite $k$-edge-connected graph has a decomposition into $k$-edge-connected countable graphs. Thus, it suffices to consider countable graphs.

The essential difference between $T$-decompositions and $T^*$-decompositions is that if $G'$ has a $T^*$-decomposition, and if $G$ can be obtained from $G'$ by vertex-identifications, then also $G$ has a $T^*$-decomposition. Since vertex-splitting is the reverse operation of vertex-identification, we would like to split the vertices in $G$ to get a locally finite graph of the same edge-connectivity. By Theorem 9 in~\cite{infinitesplitting}, there exists a splitting of the vertices of $G$ such that the resulting graph $G'$ is still $(k_h(T)+m^2-m)$-edge-connected and every block of $G'$ is locally finite. There might still be vertices of infinite degree in $G'$, but those must be cutvertices. It is not difficult to show that we can further split these cutvertices so that every connected component is $(k_h(T)+m^2-m)$-edge-connected and either locally finite, or has precisely one cutvertex of infinite degree and all its blocks are finite. In the second case, it is possible to delete some homomorphic copies of $T$ such that every block loses at most $m^2-m$ edges and has size divisible by $m$ afterwards. Now the resulting graph has a $T^*$-decomposition by the definition of~$k_h(T)$.
\end{proof}

\section*{Acknowledgement}

The author thanks Carsten Thomassen for his advice and helpful discussions, as well as Thomas Perrett for a careful reading of the manuscript.
This research was supported by ERC Advanced Grant GRACOL, project number 320812.

\bibliographystyle{plain}
\bibliography{literature}

\end{document}